\documentclass[12pt]{article}

\pdfoutput=1

\usepackage[utf8]{inputenc}
\usepackage{amsmath}
\usepackage{amssymb}
\usepackage{amsthm}
\usepackage{fullpage}
\usepackage{graphicx}
\usepackage{hyperref}
\usepackage{color}
\usepackage{algorithm}
\usepackage{algpseudocode}
\usepackage[sc]{mathpazo}
\usepackage[T1]{fontenc}
\usepackage[margin=10pt, font=small]{caption}

\newtheorem{thm}{Theorem}[section]
\newtheorem{prop}[thm]{Proposition}
\newtheorem{lem}[thm]{Lemma}
\newtheorem{coro}[thm]{Corollary}
\theoremstyle{remark}

\algdef{SE}[DOWHILE]{Do}{doWhile}{\algorithmicdo}[1]{\algorithmicwhile\ #1}%

\newcommand{\tdef}[1]{\underline{\smash{\textcolor{blue}{\emph{#1}}}}}
\newcommand{\nonsep}{\cite{nonsep}}

\newcommand{\figpath}{figures.pdf}

\author{Wenjie Fang \thanks{Partially supported by \textit{Agence nationale de la Recherche} under grant number ANR 12-JS02-001-01 ``Cartaplus''. The author is now affiliated to Institute of Discrete Mathematics, Technical University of Graz. Email: fang@math.tugraz.at} \\ Laboratoire en l'informatique du parall\'elisme \\ \'Ecole normale sup\'erieure de Lyon}
\title{Planar triangulations, bridgeless planar maps and Tamari intervals}

\begin{document}
\maketitle

\abstract{We present a direct bijection between planar 3-connected triangulations and bridgeless planar maps, which were first enumerated by Tutte (1962) and Walsh and Lehman (1975) respectively. Previously known bijections by Wormald (1980) and Fusy (2010) are all defined recursively. Our direct bijection passes by a new class of combinatorial objects called ``sticky trees''. We also present bijections between sticky trees, intervals in the Tamari lattices and closed flows on forests. With our bijections, we recover several known enumerative results about these objects. We thus show that sticky trees can serve as a nexus of bijective links among all these equi-enumerated objects.}

\section{Introduction}

Planar maps, which are embeddings of graphs on the plane, have been known to give nice enumerative formulas, which can be explained by simple bijections (\textit{cf.} \cite{Schaeffer:survey}). Furthermore, using bijections, they can be related to other objects in enumerative and algebraic combinatorics, which makes them a suitable tool for structural study.

In the enumeration of planar maps, we often consider \emph{rooted} planar maps, in which an edge called the \emph{root} is marked and given an orientation. Seemingly different classes of rooted planar maps can be enumerated by the same formula. It was first proved by Tutte \cite{tutte1962triangulation} that the number of rooted planar 3-connected triangulations with $3(n+1)$ edges (thus $n$ internal vertices) is given by
\begin{equation} \label{eq:count}
  \frac{2}{n(n+1)} \binom{4n+1}{n-1}.
\end{equation}
Later, it was proved by Walsh and Lehman \cite{walsh-lehman} that \eqref{eq:count} is also the number of rooted loopless planar maps (thus also rooted bridgeless planar maps by duality) with $n$ edges. Both proofs used functional equations. Bijective explanations were later given by Wormald \cite{wormald} and Fusy \cite{fusy-bij}, both with recursively-defined bijections.

There are other combinatorial objects enumerated by \eqref{eq:count}, such as intervals in the Tamari lattice of order $n$ \cite{ch06} and closed flows on forests with $n$ nodes \cite{chapoton-chatel-pons}. Tamari intervals are especially interesting, due to their rich structure and relation to deep algebraic structures (\textit{cf.} \cite{bergeron-preville}). Previous studies \cite{BB2009intervals, bousquet-fusy-preville} also hint a link between planar maps and intervals in Tamari-like lattices to be further explored.

Our main results are thus two-fold, one relating two classes of planar maps, the other relating planar maps to Tamari intervals. Our first result is a direct bijection between rooted bridgeless planar maps and rooted planar 3-connected triangulations, via another structure called ``sticky trees''. With different depth-first explorations, we can transform bijectively both classes of planar maps into sticky trees, which leads to our bijection. To the author's knowledge, this is the first direct bijection between these classes of planar maps.

Our second result is a series of bijections between sticky trees, Tamari intervals and closed flow on forests, which give new proofs of some known enumerative results. These bijections preserve certain structures of the classes, and can be seen as special cases of the ones in \nonsep{}. A composition of our bijections gives an alternative to the bijection from Tamari intervals to planar triangulations in \cite{BB2009intervals}. With our bijections, we can see sticky trees as a nexus for bijections between classes of objects counted by \eqref{eq:count} mentioned above. Our work thus contributes as a unification of bijective understandings about the relation between Tamari intervals and other objects, especially planar maps.

This article is organized as follows. In Section~\ref{sec:pre}, we introduce notions and definitions we need to establish our bijections. In Section~\ref{sec:bij}, we define our bijections between sticky trees, planar bridgeless map and planar triangulations, and prove their validity. We then look at the bijection from sticky trees to Tamari intervals and closed flows on forests in Section~\ref{sec:bij-oth}. We end this article with a discussion in Section~\ref{sec:dis}.

\section{Preliminaries} \label{sec:pre}

\tdef{Planar maps} are drawings of connected graphs on the sphere, defined up to orientation-preserving diffeomorphism, such that edges cross only at their common endpoint vertices. Planar regions split by edges of a planar map $M$ are called its \tdef{faces}. In map enumeration, we usually consider \tdef{rooted maps}, where an edge called the \tdef{root} is distinguished and oriented. The face on the left of the root is called the \tdef{root face}, which is drawn by convention as the face containing the point at infinity. We only consider rooted planar maps from now on.

In the following, we will concentrate on two families of planar maps: planar 3-connected triangulations (or simply \emph{planar triangulations} hereinafter) and bridgeless planar maps. A \tdef{planar triangulation} is a planar map with all faces of degree $3$. A planar triangulation is \emph{3-connected} if it has no loops or multiple edges. From now on, we only consider planar triangulations that are 3-connected. A planar triangulation always has $3n$ edges and $n+2$ vertices for $n \in \mathbb{N}_+$. A vertex $v$ in a planar triangulation is \tdef{internal} if $v$ is not adjacent to the root face. A \tdef{bridgeless planar map} is a planar map without bridges, \textit{i.e.}, edges whose deletion disconnects the map. Examples of these planar maps are given in Figure~\ref{fig:ex-map}. We denote respectively by $\mathcal{T}_n$ and $\mathcal{B}_n$ the set of planar triangulations with $n$ internal vertices and the set of bridgeless planar maps with $n$ edges. 

\begin{figure}
  \begin{center}
  \includegraphics[page=1,scale=0.65]{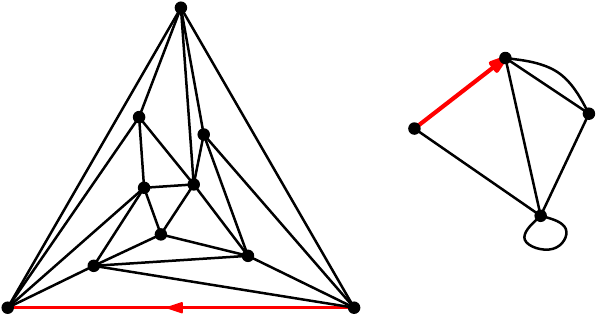}
  \end{center}
  \caption{Example of a planar triangulation and a bridgeless planar map}
  \label{fig:ex-map}
\end{figure}

To describe our bijection between $\mathcal{B}_n$ and $\mathcal{T}_n$, we need an intermediate class of objects called ``sticky trees''. Given a plane tree $S$, its \tdef{prefix order} is a total order of its nodes defined recursively: let $u$ be the root node of $S$ and $S_1, \ldots, S_k$ the sub-trees of $u$ from left to right, then the prefix order of $S$ is $u$ followed by that of $S_1$ then those of $S_2, \ldots, S_k$. The prefix order is also the order of its nodes visited in a \textbf{counter-clockwise} contour starting from its root node. The \tdef{depth} of a node in a plane tree is its distance to the root node. As a special case, the root node of a plane tree has \tdef{depth} 0. A \tdef{sticky tree} is a plane tree $S=(V,E)$ associated with a labeling function $\ell: V \to \mathbb{N}$ on its nodes satisfying the following conditions:
\begin{enumerate}
\item For a node $u$ of depth $d$, we have $0 \leq \ell(u) \leq d$.
\item For each node $u$ of depth $d > 0$, there is a node $v$ in the sub-tree rooted at $u$ such that $\ell(v) < d$ (we allow $v=u$).
\item For a node $u$ of depth $d$ and $S_u$ the sub-tree rooted at $u$, if there is a node $v$ in $S_u$ with $\ell(v)=d$ (which is the depth of $u$), then all nodes in $S_u$ (including $u$) that come before $v$ in the prefix order have a label not smaller than $d$.
\end{enumerate}

We denote by $\mathcal{S}_n$ the set of sticky trees with $n$ edges. A non-root node in a sticky tree is \tdef{primary} if its label is equal to its depth. All other non-root nodes are called \tdef{derived}. It is clear that Condition~2 only needs to be checked against primary nodes. Given a node $u$, the \tdef{certificate} of $u$ is the first node in the prefix order that makes $u$ satisfy Condition~2. Every non-root node thus has a certificate, which is itself when it is derived. We have the following lemma for the existence of primary nodes.

\begin{lem} \label{lem:primary-well-founded}
Given a node $v$ with a label $d$ in a sticky tree $S$, its ancestor $u$ of depth $d$ must be a primary node.
\end{lem}
\begin{proof}
Condition~3 on $u$ leads to $\ell(u) \geq d$, while Condition~1 on $u$ leads to $\ell(u) \leq d$.
\end{proof}

By applying Lemma~\ref{lem:primary-well-founded} to $v$ with a label $d$ in Condition~3, we can see that Condition~3 only needs to be checked against primary nodes as $u$.

The definition of sticky trees seems complicated, but we will see that it captures the class of labeled trees we obtain from bridgeless planar maps by an exploration process that we will now define. Readers may have already noticed that we reserve the term \emph{node} for trees and \emph{vertex} for maps to distinguish where these objects live for clarity. Our definition of sticky trees is inspired by that of decorated trees in \nonsep{}, and their relations will be discussed in Section~\ref{sec:bij-oth}.

\section{Bijections with planar maps} \label{sec:bij}

In this section, we present two bijections between sticky trees and two classes of planar maps: bridgeless planar maps and planar triangulations. Together they form a direct bijection between these two families of planar maps.

\subsection{From bridgeless planar maps to sticky trees}

We first describe an exploration process of planar maps that gives rise to a bijection between bridgeless planar maps and sticky trees. During this process, vertices in a map will be duplicated, but the number of edges will stay the same. 

Given a planar map, we can cut all its edges in the middle, and we obtain a set of vertices with attached \tdef{half-edges}. For a planar map $M$ with its root $e_r$ pointing from $v$ to $u$, we denote by $h_r$ the \tdef{root half-edge}, which is the half-edge of $e_r$ adjacent to $v$. By definition, each edge $e$ gives two half edges $h_{e,1}, h_{e,2}$, and we define the involution $\tau$ on the set of half edges such that $\tau(h_{e,1}) = h_{e,2}$ for all $e$. We define a bijection $\sigma$ on the set of half-edges such that, for any half-edge $h$, its image $\sigma(h)$ is the next half-edge sharing the same vertex as $h$ in \textbf{clockwise} order.

\begin{figure}
\begin{center}
  \begin{minipage}{0.8\textwidth}
    \begin{algorithm}[H]
      \caption{Exploration algorithm for bridgeless planar maps}
      \label{algo:dfs}
      \begin{algorithmic}
        \Function{\texttt{Explore}}{Half-edge $h$}
        \State Mark $h$ and $\tau(h)$ as visited
        \State $h' \gets \sigma(\tau(h))$
        \While{$h'$ is not visited}
        \State \texttt{Explore}(h')
        \State $h' \gets \sigma(h')$
        \EndWhile
        \EndFunction
      \end{algorithmic}
    \end{algorithm}
  \end{minipage}
  \par
\end{center}
\end{figure}

We now define an exploration algorithm of half-edges, starting from the root half-edge $h_r$ (see Figure~\ref{fig:bij-bridgeless} for an example of its execution). When we explore a half-edge $h$, we first mark $h$ and $\tau(h)$ as already visited, then we try to explore in depth-first manner the half-edges next to $\tau(h)$ in \textbf{clockwise} order, until we meet a half-edge that is already visited upon inspection, and we finish the exploration of $h$. A pseudo-code of the algorithm is provided in Algorithm~\ref{algo:dfs}.

After the exploration, we draw the exploration tree $S$, whose nodes are some of the half-edges in $M$. We will identify nodes in $S$ as ``copies'' of vertices in $M$, where a half-edge is regarded as its adjacent vertex. We now label both vertices in $M$ and nodes in $S$. For a vertex $w$ in $M$, we find the node $w'$ in $S$ corresponding to the first visit of $w$, and we label both by the depth of $w'$ in $S$. The labels of other nodes of $S$ are those of their corresponding vertices in $M$. We denote by $\mathrm{S}(M)$ the labeled tree we get after the exploration and labeling process above.

\begin{figure}
  \begin{center}
    \includegraphics[page=2,scale=0.85]{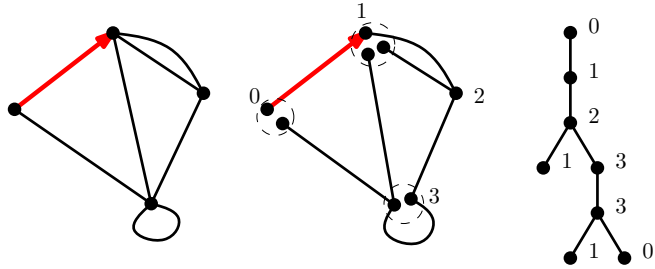}
  \end{center}
  \caption{An example of the depth-first exploration of edges on a bridgeless map}
  \label{fig:bij-bridgeless}
\end{figure}

By construction, $\mathrm{S}(M)$ is a plane tree with all nodes labeled, and it has the same number of edges as $M$. We also notice that $\mathrm{S}(M)$ is well-defined for all planar maps. We have the following observations on plane trees obtained using $\mathrm{S}$.

\begin{prop} \label{prop:b-to-s-labels}
  Let $B$ be a planar map. Then, given a node $u$ in $\mathbb{S}(B)$ with a label $d$ strictly smaller than its depth, its ancestor $u_0$ of depth $d$ must be a primary node.
\end{prop}
\begin{proof}
  We first observe that, for $v_B$ a vertex in $B$, and $v$ the node in $\mathrm{S}(B)$ corresponding to the first visit to $v_B$ in the exploration, any node in $\mathrm{S}(B)$ corresponding to a subsequent visit to $v_B$ in $B$ is a descendant of $v$. The reason is that, once the exploration process reaches the vertex $v_B$ on $B$, that is, reaching $v$ on $\mathrm{S}(B)$, it will not backtrack from $v$ until every half-edge of $v_B$ is visited. Indeed, if $h$ is a half-edge of $v_B$ that we visited, then when the exploration backtracks to $h$, the next half-edge $\sigma(h)$ is either already visited, or the exploration will continue on $\sigma(h)$.

  We now consider our claim. By definition, $u$ corresponds to a visit to a certain vertex $u_B$ in $B$, but not the first visit. Let $u_0$ be the node in $\mathbb{S}(B)$ corresponding to the first visit to $u_B$. We know that $u_0$ has the label $d$. Then by the definition of exploration, the label of $u_0$ is its depth in $\mathbb{S}(B)$, which means that $u_0$ is a primary node. By the first point, $u_0$ is also an ancestor of $u$ of depth $d$, which concludes the proof.
\end{proof}

\begin{prop} \label{prop:explore-order}
  Let $B$ be a planar map and $u_1, u_2$ two nodes in $\mathbb{S}(B)$ with the same label $d$, with $u_1$ coming before $u_2$ in the prefix order. Suppose that $u$ is the corresponding vertex in $B$, and $h_1$, $h_2$ the half-edges of $u$ leading to visits corresponding to $u_1, u_2$. Now let $h_0$ be the first half-edge of $u$ visited in the exploration. Then, in the clockwise order starting from $h_0$, we also have $h_1$ coming before $h_2$.
\end{prop}
\begin{proof}
  Suppose that it is not the case, that is, $h_2$ comes before $h_1$ in clockwise order starting from $h_0$. Since $u_1$ comes before $u_2$ in the prefix order, either $u_1$ is an ancestor of $u_2$ in $\mathbb{S}(B)$, or $u_2$ is not an ancestor of $u_1$ in $\mathbb{S}(B)$, but occurs before $u_1$ in the exploration.

  For the case of $u_1$ being an ancestor of $u_2$, the path from $u_1$ to $u_2$ on the tree corresponds to a cycle $C$ on $B$, starting from some half-edge $h_3$ explored in the exploration of $h_1$ and ending at $h_2$. Now, by the exploration process, $h_3$ must come after $h_1$, meaning that we have $h_0, h_2, h_1, h_3$ in clockwise order. Therefore, $C$ separates $h_1$ and $h_0$ by planarity. Hence, the exploration starting from $h_0$ cannot reach $h_1$ before visiting some half-edge on $C$, which would lead to a visit of $h_3$ before that of $h_1$, which is impossible.
  
  For the other case, let $v$ be the lowest common ancestor of $u_1$ and $u_2$, with $h^v_0$ the half-edge that links $v$ to its parent in the exploration, and $h^v_1$, $h^v_2$ the half-edges that leads to $u_1$ and $u_2$ respectively in the exploration. By the exploration process, we must have $h^v_0, h^v_2, h^v_1$ in clockwise order, since $u_1$ comes before $u_2$ in prefix order. Therefore, the cycle $C$ on $B$ corresponding to the path from $u_0$ to $u_2$ via $v$ separates $h^v_1$ and $h_1$. By the same argument as in the previous case, it is impossible to have a path from $h^v_1$ to $h_1$ in the exploration, leading to a contradiction. The same reasoning also holds when $v=u$. We thus conclude the proof.
\end{proof}

When the map is bridgeless, its image by $\mathrm{S}$ has the following property.

\begin{prop} \label{prop:b-to-s}
If $B \in \mathcal{B}_n$ is a bridgeless planar map, then $\mathrm{S}(B)$ is a sticky tree in $\mathcal{S}_n$.
\end{prop}
\begin{proof}
  We only need to prove that $\mathrm{S}(B)$ satisfies the conditions of sticky trees. Condition~1 comes directly from the definition of $\mathrm{S}(B)$. We now deal with Conditions 2 and 3.

  For Condition~2, we suppose that Condition~2 does not hold for some node $w$ with a label $d$ in $\mathrm{S}(B)$, that is, every descendant of $w$ in $\mathrm{S}(B)$ has a label at least $d$. We can suppose that $w$ is primary, since derived nodes satisfy Condition~2 automatically. Let $w_B$ be the vertex that corresponds to $w$ in $B$, $e_B$ the edge from which $w_B$ was first visited in the exploration, and $S_w$ the sub-tree of $\mathrm{S}(B)$ rooted at $w$. We consider the set of vertices $V_w$ of $B$ that are first visited during the construction of $S_w$, which corresponds to primary nodes in $S_w$. For derived nodes, let $u$ be a derived node with label $d'$, and we have $d' \geq d$. Then by Proposition~\ref{prop:b-to-s-labels}, its corresponding primary node is the ancestor of depth $d' \geq d$, which is still a primary node in $S_w$. Therefore, $u$ corresponds to a visit to some vertex in $V_w$. Therefore, by the completeness of exploration as explained in the proof of the first point of Proposition~\ref{prop:b-to-s-labels}, the only edge adjacent to a vertex in $V_w$ that links to a vertex not in $V_w$ is $e_B$, which means $e_B$ is a bridge, contradiction $B \in \mathcal{B}_n$. We conclude that Condition~2 is satisfied for $\mathrm{S}(B)$.

  For Condition~3, for $u$ a node in $\mathrm{S}(B)$ of depth $d$ and $v_1, v_2$ two descendants of $u$ such that $v_1$ comes before $v_2$ in the prefix order and $\ell(v_2)=d$, we need to prove that $\ell(v_1) \geq d$. By Proposition~\ref{prop:b-to-s-labels}, $\ell(v_2)=d$ implies that $u$ is primary, and we denote by $u_B$ the vertex in $B$ corresponding to $u$. We now track several related objects on $B$ and $\mathrm{S}(B)$, which are illustrated in Figure~\ref{fig:m-to-s-cond3}. Let $S_u$ be the sub-tree of $\mathrm{S}(B)$ rooted at $u$, and $h_u$ the half-edge that leads the first time to $u_B$ on $B$ in the exploration. We order edges around $u_B$ clockwise, starting from $h_u$. Let $v$ be the lowest common ancestor of $v_1$ and $v_2$, and $P$ the path from $u$ to $v_2$ on $\mathrm{S}(B)$. The node $v$ is clearly on the path $P$. Let $u_1$ (resp. $u_2$) be the node with label $d$ that comes just before $v$ (resp. after $v$) in the prefix order. Let $P'$ be the path from $u_1$ to $u_2$ in $\mathrm{S}(B)$, and $P'_B$ its counterpart in $B$. We notice that $P'$ may not be always descending in terms of depth. Let $h_1$ (resp. $h_2$) be the half-edge on $B$ corresponding to the first (resp. the last) half-edge of $P'_B$ (again, see Figure~\ref{fig:m-to-s-cond3}). We now consider the cycle $C$ in $B$ corresponding to the path from $u$ to $u_2$. Since $u_1$ comes before $u_2$ in the prefix order, by Proposition~\ref{prop:explore-order}, $h_2$ comes after $h_1$ in clockwise order around $u_B$, and $P'_B$ clockwise encloses a region $D$ in $B$ away from $h_u$. Again by the exploration process, since $v_1$ comes before $v_2$, it must be enclosed in $D$ away from ancestors of $u$ (see the right side of Figure~\ref{fig:m-to-s-cond3}). Therefore, $\ell(v_1) \geq d$, and Condition~3 is satisfied.
\end{proof}

\begin{figure}
  \begin{center}
    \includegraphics[scale=0.85,page=8]{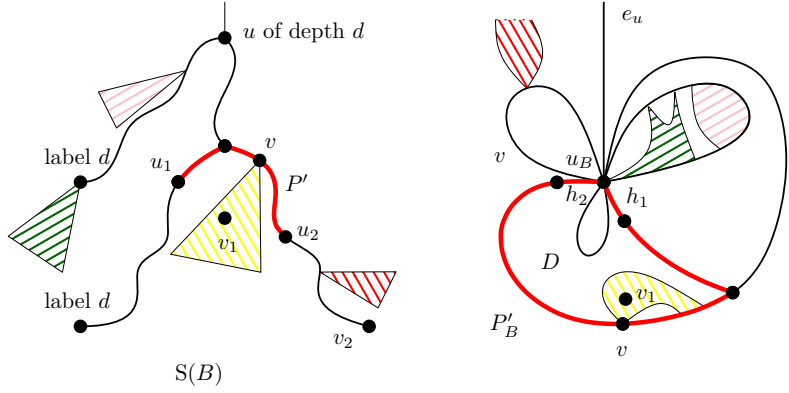}
  \end{center}
  \caption{The path between a node and its descendant with the same label on $\mathrm{S}(B)$}
  \label{fig:m-to-s-cond3}
\end{figure}

From the proof above, we can see how the three conditions of sticky trees characterize an exploration tree of bridgeless planar maps. Condition~1 ensures that a sticky tree is an exploration tree of a certain map, while Condition~2 ensures that the map is bridgeless, and Condition~3 ensures that the map is planar.

We now describe a procedure $\mathrm{R}$ that converts a sticky tree into a bridgeless planar map. An example can be found in Figure~\ref{fig:bij-sticky}. Given a sticky tree $S$, we put a plug on the left of each node. Then for a node $u$ with a label $d$ that is not derived, we glue in clockwise fashion the plugs of $u$ and descendants of $u$ with a label $d$ in prefix order to form a vertex. We notice that the positions of plugs make all descending half-edges of a node $u$ follow the ascending half-edge that links $u$ to its parent in clockwise order around the vertex obtained after gluing. By Lemma~\ref{lem:primary-well-founded}, each derived node has a corresponding primary node to glue. After gluing all primary nodes in increasing order of labels, we obtain a map denoted by $\mathrm{R}(S)$ with no plug left, rooted at the right-most edge of the root node, pointing away from the root node. The process fails if all plugs corresponding to a primary node are not in the same face at any point of the process. The same gluing process is applied to all primary nodes in increasing order of depth.

\begin{figure}
  \begin{center}
    \includegraphics[page=3,scale=0.85]{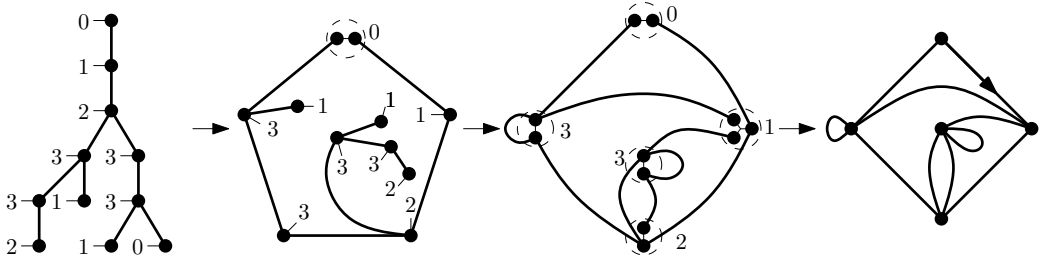}
  \end{center}
  \caption{Node-gluing procedure to get a bridgeless planar map from a sticky tree}
  \label{fig:bij-sticky}
\end{figure}

By construction, if succeeded, $\mathrm{R}(S)$ is a planar map. We have the following result.

\begin{prop} \label{prop:s-to-b}
For $S$ a sticky tree in $\mathcal{S}_n$, the planar map $\mathrm{R}(S)$ is always defined and is in $\mathcal{B}_n$.
\end{prop}
\begin{proof}
  The number of edges is clearly preserved. We first prove that the gluing process never fails. Suppose that a failure occurs at the primary node $u$ with a label $d$, after gluing another primary node $v$ with a label $d' \leq d$. If $v$ is not an ancestor of $u$, then the gluing of $v$ cannot affect that of $u$. Therefore, $v$ must be an ancestor of $u$, and $d' < d$. In this case, for plugs gluing to $u$ to be cut apart, there must be a node $v'$ labeled by $d'$ in the sub-tree $S_u$ rooted by $u$, preceding a node $u'$ with a label $d$ in $S_u$ in prefix order, which violates Condition~3. Therefore, $\mathrm{R}(S)$ is always defined. If there is a bridge $e$ in $\mathrm{R}(S)$, then $e$ must link a primary vertex $u$ in $S$ to its parent, and $u$ must violates Condition~2, which is a contradiction. Therefore, $\mathrm{R}(S)$ is bridgeless.
\end{proof}

With some detailed analysis, we can prove that $\mathrm{R}$ is in fact the inverse of $\mathrm{S}$, which makes them both bijections.

\begin{thm} \label{thm:b-s-bij}
For all $n \geq 1$, the transformation $\mathrm{S}$ is a bijection from $\mathcal{B}_n$ to $\mathcal{S}_n$, and $\mathrm{R}$ is its inverse.
\end{thm}
\begin{proof}
  Since the size parameter is clearly conserved by the two bijections, we only need to prove that $\mathrm{S}(\mathrm{R}(S))=S$ for every sticky tree $S$, and $\mathrm{R}(\mathrm{S}(B))=B$ for every bridgeless planar map $B$.
  
  We first prove that $\mathrm{S}(\mathrm{R}(S))=S$. We observe that the edge exploration can be done on arbitrary planar maps, thus we only need to prove that the tree structure of the exploration tree is preserved in the intermediate planar maps after each plug-gluing step.

  We start by laying out some definitions. Without ambiguity, we refer to corresponding edges in all intermediate maps by the same name. Let $M_1$ be an intermediate map in the construction of $\mathrm{R}(S)$, and $u_0$ a primary node or the root in $S$ of depth $d$ that has not yet been glued up. As illustrated in Figure~\ref{fig:bij-glue}, we denote by $u_1, u_2, \ldots, u_k$ the descendants of $u$ with a label $d$ in the prefix order, by $h_0, \ldots, h_k$ the half-edges linking each $u_i$ to its parent in $S$, and by $L_0, \ldots, L_k$ the lists of descendant half-edges of each $u_i$ \emph{from right to left}. Here, we temporarily suspend the case where $u_0$ is the root vertex. Let $M_2$ be the map obtained by gluing the plugs of $u_0, u_1, \ldots, u_k$ into a vertex $u_*$. Since a node always have its plug preceding its descendants, the half-edges around $u_*$ in clockwise order are $h_0, L_0, h_1, L_1, \ldots, h_k, L_k$. Since other vertices are not altered, the first visit of $u_*$ still comes by $h_0$.

  Suppose that the explorations on $M_1$ and $M_2$ act differently, with the last common half-edge $h_*$, which must be marked visited during $\mathrm{Explore}(\tau(h_*))$, otherwise $\tau(h_*)$ would also be common half-edge. Since $M_1$ and $M_2$ only differ on the glued vertex $u_*$, the half-edge $h_*$ must be adjacent to $u_*$ on $M_2$. Since the explorations of $M_1$ and $M_2$ agree up to $h_*$, we must have $h_* = h_i$ for a certain $i$. If $L_i$ is not empty, then the next half-edge to visit after $h_*$ on both $M_1$ and $M_2$ will be the same, which is the first half-edge of $L_i$, but this situation violates the maximality of $h_*$. Therefore, $L_i$ is empty, and the exploration in $M_1$ will backtrack. Now, if $i=k$, since $h_0$ must have been visited, the exploration in $M_2$ will backtrack too, and ending in the same half-edge as in $M_1$. Therefore, $i \neq k$. Since $L_i$ is empty, $u_i$ has no descendant in $S$, and $u_{i+1}$ is not a descendant of $u_i$ in $S$, which means that $h_{i+1}$ is visited before $h_*=h_i$ in $M_1$, and equally in $M_2$. The exploration in $M_2$ after visiting $h_*$ must backtrack, and since the explored edges are the same in both $M_1$ and $M_2$, both explorations will backtrack to the same half-edge. Thus, the next edge visited in both $M_1$ and $M_2$ will be the same, which violates again the maximality of $h_*$. Therefore, $M_1$ and $M_2$ have the same explorations on edges, meaning that the plug-gluing of a primary node preserves the tree structure of the exploration tree, leading to $\mathrm{S}(\mathrm{R}(S))$ sharing the same tree structure with $S$. We notice that the former argument also works for the root node, since the role of $h_0$ is to fix the first half-edge to visit, and we know that for the root node. For labels, we observe simply that all primary nodes of $\mathrm{S}(\mathrm{R}(S))$ receives the same label as in $S$, and all derived nodes are thus also correctly labeled by the constructions. We conclude that $\mathrm{S}(\mathrm{R}(S))=S$.

\begin{figure}
  \begin{center}
    \includegraphics[page=9,scale=0.85]{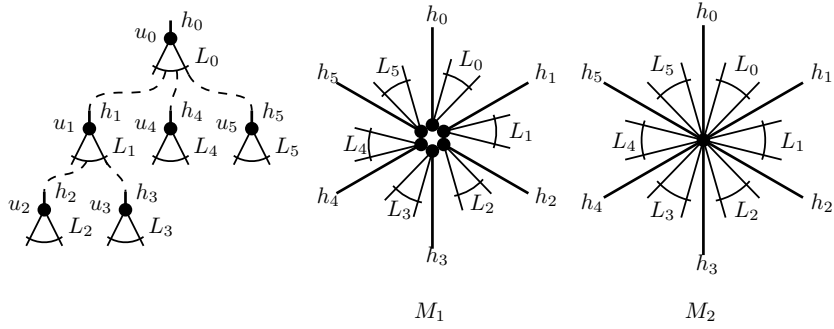}
  \end{center}
  \caption{Example of half-edges around a glued vertex $u_*$}
  \label{fig:bij-glue}
\end{figure}
  
  We now prove that $\mathrm{R}(\mathrm{S}(B)) = B$ for every bridgeless map $B$. We first observe that $\mathrm{R}(\mathrm{S}(B))$ has the same underlying graph as $B$. Let $u_0$ be a primary node or the root node of $\mathrm{S}(B)$ of depth $d$. We now prove that there is only one way to glue $u_0$ with its descendants labeled by $d$ such that the result is still planar with the same exploration tree. Let $u_1, \ldots, u_k$ be the descendants of $u$ with a label $d$. As before (see Figure~\ref{fig:bij-glue}), for $0 \leq i \leq k$, we denote by $h_i$ the half-edge of $u_i$ leading to its parent, and by $L_i$ the list of descendant half-edges of $u_i$ \emph{from right to left}. Again, the case where $u_0$ is the root node can be treated similarly by the same argument as before. Let $u_*$ be the vertex after gluing all $u_i$. We give half-edges adjacent to $u_*$ an order that starts from $h_0$ and goes in clockwise direction. From the exploration process, to preserve the exploration tree after gluing, half-edges in $L_i$ must follow $h_i$ around $u_*$ in the order of $L_i$. Therefore, the order of half-edges around $u_*$ must take the form $h_0, L_0, h_{i_1}, L_{i_1}, \ldots, h_{i_k}, L_{i_k}$. It is then clear that the only such order that does not introduce any crossing is $h_0, L_0, h_1, L_1, \ldots, h_k, L_k$. Therefore, by the uniqueness of gluing, the order of half-edges around each vertex of $\mathrm{R}(\mathrm{S}(B))$ is the same as $B$, and we have $\mathrm{R}(\mathrm{S}(B))=B$.
\end{proof}

\subsection{From planar triangulations to sticky trees} \label{sec:bij-tri}

\begin{figure}
  \begin{center}
    \includegraphics[page=4,scale=0.75]{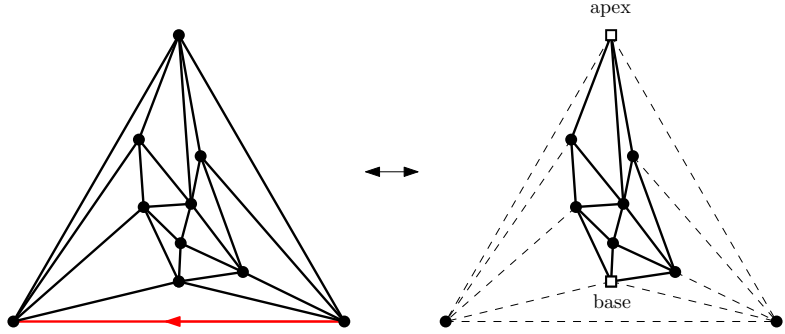}
  \end{center}
  \caption{A triangulation and its core}
  \label{fig:tri-core}
\end{figure}

We now present a bijection between planar triangulations and sticky trees. The \tdef{core} $C(T)$ of a triangulation $T$ rooted at $e=(u,v)$ is obtained by deleting $u,v$ and all their adjacent edges from $T$. Figure~\ref{fig:tri-core} shows an example. The \tdef{apex} of $C(T)$ is the remaining vertex of the outer face, and the \tdef{base} is the remaining vertex of the triangle on the right of the root. The \tdef{left (resp. right) boundary} of $C(T)$ is the leftmost (resp. rightmost) path from the apex to the base (inclusive). Since $T$ has no double edges, the two boundaries are simple, \textit{i.e.}, without repeated vertices. However, the two boundaries may share some vertices. To recover $T$ from $C(T)$, we put two vertices $u,v$ and link all vertices on the left (resp. right) boundary to $u$ (resp. $v$), then draw from $v$ to $u$ the root $e$ such that the apex is outside. For $T \in \mathcal{T}_n$, its core $C(T)$ has $n+1$ vertices.

\begin{figure}
  \begin{center}
    \includegraphics[page=5,scale=0.8]{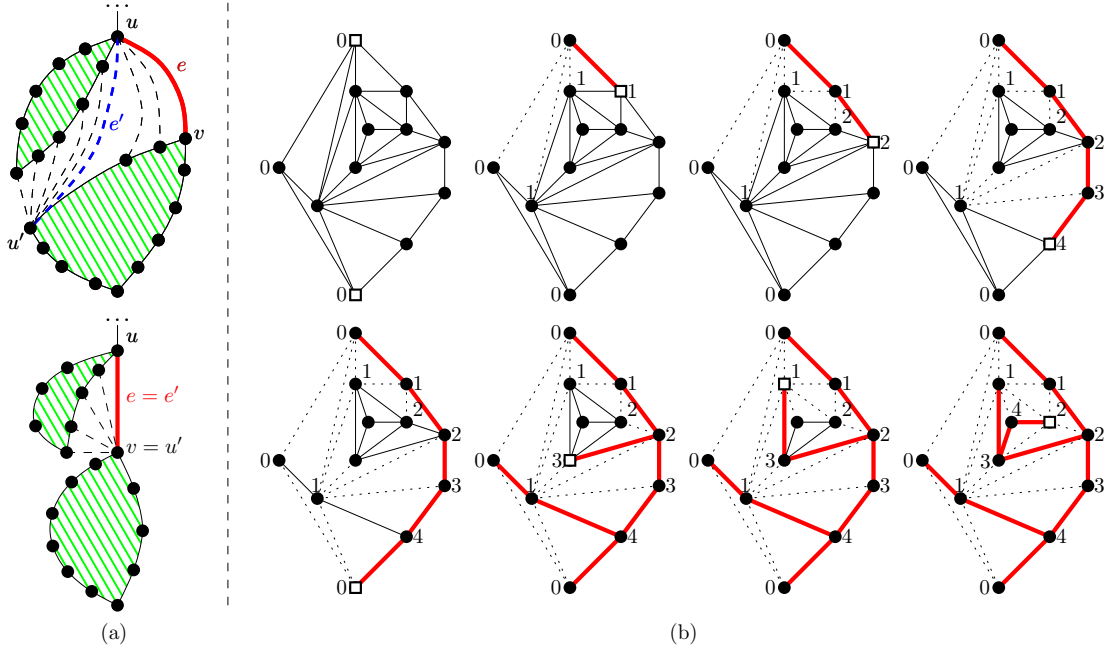}
  \end{center}
  \caption{(a) Unlabeled (above) and labeled (below) cases of edge deletion, where dashed lines are edges deleted after visiting $e$ (b) Example of exploration on the core of a planar triangulation}
  \label{fig:edge-delete}
\end{figure}

We now describe a clockwise depth-first exploration of vertices with edge deletion on the core $C(T)$ of a planar triangulation $T$ to obtain a labeled tree. We first label all vertices on the left boundary by $0$ (including apex and base). Then we start the exploration from the apex by its adjacent edge on the right boundary. Suppose that we arrive at a vertex $u$ of depth $d-1$ in the exploration, and we are instructed to start from a certain edge $e_0$ adjacent to $u$. We start by exploring edges adjacent to $u$ in clockwise order, starting from $e_0$. Let $e$ be such an edge explored, which links $u$ to a new vertex $v$. Upon exploration of $e$, we delete some edges in the way illustrated in Figure~\ref{fig:edge-delete}(a). Here is a detailed description. If $v$ has no label yet, we label it by $d$, and we delete consecutive edges adjacent to $u$ that come after $e$ in clockwise order, until reaching an edge $e'$ that links $u$ to a labeled vertex $u'$. These deleted edges link $u$ to unlabeled vertices, which are then labeled by $d$. If $v$ is already labeled, we set $u'=v$ and $e'=e$. We then delete edges adjacent to $u'$ that comes after $e'$ in counter-clockwise order, until $e$ becomes a bridge (if impossible, the process fails). In the end, if $e' \neq e$, we also delete $e'$. After the deletions, we recursively explore $v$, starting from the edge adjacent to $v$ that is next to $e$ in clockwise order. After the exploration of $v$ is finished, we look for the next (not-yet-deleted) edge adjacent to $u$ next to $e$ in clockwise order, and perform the same procedure as for $e$. The exploration of $u$ finishes when all possibilities are exhausted. Figure~\ref{fig:edge-delete}(b) shows an example of this process, where the partial exploration tree and the deleted edges are shown. Since the exploration always leaves a visited vertex labeled and a visited edge becoming a bridge, in the end we obtain from $T$ a tree with all nodes labeled, denoted by $\mathrm{P}(T)$.

The exploration process above has the following properties.
\begin{lem} \label{lem:triang-expl}
In the exploration of a planar triangulation $T$, whenever we arrive at some new vertex $u$ before any failure, if we delete all explored edges, we obtain several connected components. Given a component $C$, let $v$ be the vertex of $C$ closest to the apex in the exploration tree, which is on the boundary by planarity. The vertex $v$ is called the \tdef{lead vertex} of $C$. Then $C$ satisfies the following properties:
\begin{enumerate}
\item All labeled vertices in $C$ are on the boundary;
\item The labeled vertices of $C$ form a consecutive simple chain $P$ starting from $v$ with decreasing labels in counter-clockwise order;
\item The label of $v$ is not larger than its distance to the apex;
\item $C$ without labels can be regarded as the \emph{core of a triangulation}, with $v$ the apex and $P$ the left boundary.
\end{enumerate}
\end{lem}
\begin{proof}  
  At the beginning of the exploration, there is only one component $C(T)$ with its vertices on the left boundary labeled by $0$, which satisfies all the properties. We now prove that these properties are kept after exploring an edge. From the exploration process (see Figure~\ref{fig:edge-delete}(a)), when an edge $\{u, v\}$ is explored, we split a component $C$ into two, the upper one $C_1$ and the lower one $C_2$. We suppose that $C$ satisfies all properties above, and we only need to prove that $C_1$ and $C_2$ also satisfy these properties, with lead vertices $u$ and $v$ respectively. Let $d$ be the distance of $u$ to the apex of $T$ in the exploration tree.

  For Property 1, by construction, all labeled vertices in both $C_1$ and $C_2$ are on their boundaries.
  
  Let $P$ be the chain of labeled vertices in $C$. Property 2 on $C_1$ is clear, since its labeled vertices form a segment of the chain $P$, starting from its lead vertex $u$. For the Property 2 on $C_2$, by Property 2 and 3 applied to $C$, all vertices in $C$, which include those in $C_2$, have labels at most $d$. Only when $v$ is not yet labeled before exploring $\{u, v\}$ do we have newly labeled vertices, which are all labeled by $d+1$. These newly labeled vertices precede already labeled vertices on the portion of $P$ on $C_2$ to form a new path $P'$ with decreasing labels in clockwise order. Therefore, the second property also holds on $C_2$.

  For Property 3, since it is also satisfied by $C$, we have $\ell(u) \leq d$, which means that $C_1$ also satisfies Property 3. For $C_2$, its lead vertex $v$ has distance $d+1$ to the apex. If $v$ has no label before the exploration, it receives a label $\ell(v)=d+1$. Otherwise, combining Properties 2 and 3 on $C$, we have $\ell(v) \leq \ell(u) \leq d$. In both cases, we have $\ell(v) \leq d+1$, which makes $C_2$ satisfy Property 3.

  For Property 4, since it is satisfied by $C$, all internal faces of $C$, thus those of $C_1$ and $C_2$, are triangles. Furthermore, the boundaries of $C_1$ and $C_2$ are simple, because they are either a portion of a boundary of $C$, or created from edge deletion that involves no double edges. Therefore, Property 4 is also satisfied by both $C_1$ and $C_2$.
\end{proof}

Property~4 of Lemma~\ref{lem:triang-expl} can also be seen as a recursive decomposition of cores of triangulations. We also have the following simple observation on the exploration process.

\begin{prop} \label{prop:P-descendant}
  Let $T$ be a planar triangulation, and $u$ its vertex. Suppose that the exploration on $T$ does not fail, meaning that $\mathrm{P}(T)$ is well-defined. Then the descendants of $u$ in $\mathrm{P}(T)$ are exactly those in the component $C_u$ with $u$ as apex on the first visit of $u$.
\end{prop}
\begin{proof}
  It is clear that, at the first visit of $u$, the edge $e$ that leads to this visit was already a bridge, and no other vertex in $C_u$ other than $u$ was visited yet. Since the exploration does not fail, every vertex of $C_u$ is visited during the exploration of $u$, making them descendant of $u$ in $\mathrm{P}(T)$.
\end{proof}

Using these facts, we can prove that the image of $\mathrm{P}$ is the set of sticky trees. 

\begin{prop} \label{prop:t-to-s}
Let $T$ be a planar triangulation with $n$ internal vertices. Then $\mathrm{P}(T)$ is always defined and is a sticky tree with $n$ edges.
\end{prop}
\begin{proof}
  The size parameter is clearly preserved. To prove that $\mathrm{P}(T)$ is always defined, we need to prove that the exploration never fails. We now use some terminologies in Lemma~\ref{lem:triang-expl}. Suppose that the exploration arrives at a vertex $u$ without failure. We consider the component $C$ with $u$ as lead vertex, which is well-defined since there is no failure up to $u$. We can thus apply Lemma~\ref{lem:triang-expl} to $C$. We now consider the exploration of $v$ from $u$ through the edge $e$. The vertex $v$ is either labeled or unlabeled upon visit. If $v$ is labeled, by the exploration process and Properties~2~and~4 of Lemma~\ref{lem:triang-expl}, it must be on both the left and the right boundaries of $C$, thus a cut vertex of $C$, and $e$ can be made a bridge by edge-deletion. If $v$ is unlabeled, then $u'$ must be on the left boundary of $C$ again by Properties~2~and~4, and by planarity, the edge $e'$ cuts $C$ into two parts, meaning that the edge-deletion makes $e$ a bridge. The exploration can thus continue without failure.

  It remains to prove that $\mathrm{P}(T)$ is a sticky tree. Condition~1 comes from Property 3 in Lemma~\ref{lem:triang-expl} applied to $\mathrm{P}(T)$. For Condition~2, given a non-root node $v$ in $\mathrm{P}(T)$, we consider the first moment the exploration visits $v$ in $T$ from another vertex $u$ of distance $d$ to the apex of $T$. If $v$ is already labeled, by Properties 2 and 3 in Lemma~\ref{lem:triang-expl} applied to the component of $u$, the label of $v$ must be at most $d$, which is one less than the distance of $v$ to the apex of $T$. We remark that $v$ corresponds to a derived node in $\mathrm{P}(T)$ in this case. If $v$ is not yet labeled, by the exploration process and Proposition~\ref{prop:P-descendant}, the vertex $u'$ will be visited from a descendant of $v$, and $u'$ has a label at most $d$. Condition~2 is also satisfied in this case.

  For Condition~3, we consider a non-root node $u$ of depth $d$ in $\mathrm{P}(T)$. Let $C_u$ be the component with $u$ the lead vertex upon the first visit of $u$ in the exploration. The descendants of $u$ in $\mathrm{P}(T)$ are exactly vertices in $C_u$, as in Proposition~\ref{prop:P-descendant}. It is clear that all new labels must be at least $d+1$. Therefore, a descendant $v_1$ of $u$ with a label $d$ must have already been labeled in $C_u$. By Property~2 in Lemma~\ref{lem:triang-expl}, for a vertex $v_2$ with a label strictly less than $d$, it can only come after $v_1$ in counter-clockwise order on the left boundary. We thus only need to prove that $v_2$ also comes after $v_1$ in prefix order in $\mathrm{P}(T)$. Let $w$ be the lowest common ancestor of $v_1$ and $v_2$ in $\mathrm{P}(T)$. If $w$ is equal to one of $v_1$ and $v_2$, then by their order on the left boundary of $C_u$, we must have $w=v_1$, and $v_2$ indeed comes after $v_1$ in prefix order. We now suppose that $w$ is different from either $v_1$ or $v_2$. Since $C_w$ contains $v_1$ and $v_2$ by Proposition~\ref{prop:P-descendant}, it must also contains in its left boundary a segment $S$ of the left boundary of $C_u$ from $v_1$ to $v_2$. Let $w_1$ (resp. $w_2$) be the child of $w$ in $\mathrm{P}(T)$ whose exploration leads to $v_1$ (resp. $v_2$). If $w_1$ comes after $w_2$ in the prefix order, then $w_1$ will be visited first in the exploration. As in the proof of Lemma~\ref{lem:triang-expl}, the exploration process will break $S$ into two parts, $S_1$ and $S_2$ in clockwise order, with $C_{w_1}$ containing the latter part $S_2$. However, since $v_1$ precedes $v_2$ on $S$, we know that $v_1 \in S_1$ and $v_2 \in S_2$, which leads to a contradiction to the definition of $w_1$ and $w_2$. Therefore, $w_1$ comes before $w_2$ in the prefix order in $\mathrm{P}(T)$, meaning that $v_1$ precedes $v_2$ also in the prefix order. Therefore, Condition~3 is satisfied.
\end{proof}

We now describe the reverse direction from sticky trees to triangulations. Let $S$ be a sticky tree. We perform the following ``triangulation reconstruction'' as illustrated in Figure~\ref{fig:rejoin-tree}. For a non-root node $v$ of depth $d$ in $S$, let $u$ be its parent, $v'$ its certificate, $P$ the right-most branch of the sub-tree rooted at the child of $u$ that precedes $v$ in the prefix order (just to the left of $v$), and $e$ the edge from $v$ to $u$. When $v$ is primary, we link $u$ to the left of descendants of $v$ with a label $d$ in a non-crossing fashion. We also link $u$ to $v'$ by an edge $e'$, and then $v'$ to the right of nodes on $P$, if $P$ exists, with edges in counter-clockwise order around $v'$ after $e'$. When $v$ is derived, we link $v$ itself to the right of nodes on $P$ in the same way as the previous case when $P$ exists. We say that the new edges adjacent to $u$ in the case where $v$ is primary are of \tdef{type A}, and all other new edges are of \tdef{type B}. We observe that type A edges link a node to one of its descendants, while type B edges do not. This procedure is performed from the last node in the prefix order up to the first one. We denote by $\mathrm{Q}'(T)$ the map we get. We have the following proposition.

\begin{figure}
  \begin{center}
    \includegraphics[page=6,scale=0.92]{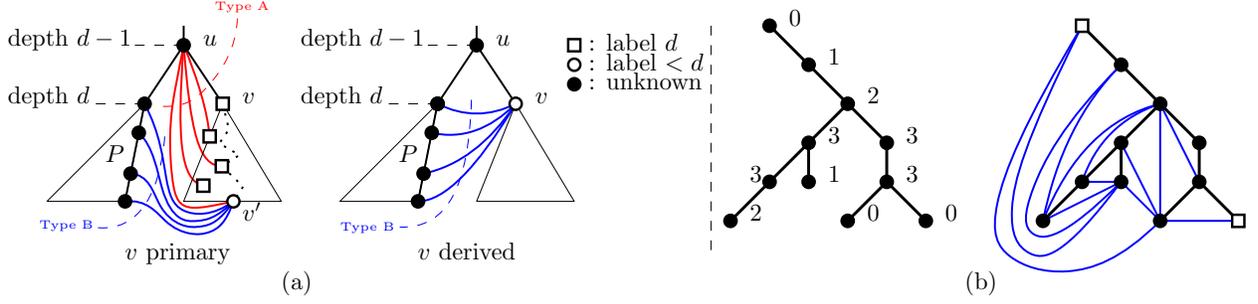}
  \end{center}
  \caption{(a) Two cases of adding edges in $\mathrm{Q}'(S)$ (b) Example of reconstruction}
  \label{fig:rejoin-tree}
\end{figure}

\begin{prop} \label{prop:s-to-t}
Let $S \in \mathcal{S}_n$ be a sticky tree with $n$ edges. The map $\mathrm{Q}'(S)$ obtained from $S$ is the core of a planar triangulation (denoted by $\mathrm{Q}(S)$), with is apex the root of $S$ and its base the last node in $S$ in the prefix order.
\end{prop}
\begin{proof}
  We first prove that $\mathrm{Q}'(S)$ has no multiple edge nor loop. By definition, it is clear that $\mathrm{Q}'(S)$ has no loop, and new edges do not double with edges in $S$. We now discuss new edges by their types. Recall that there are two types of new edges: type A that link a node to one of its descendants, and type B that do not. If there are multiple edges, they must be of the same type. Type A edges come from linking the parent $u$ of a primary node $v$ to descendants of $v$, which cannot be double because each descendant of $v$ links to $u$ only once. Type B edges come from linking the certificate of a node $v$ to the right-most path $P$ that comes just before $v$. Similarly, they cannot be double because $P$ is always simple.

  Secondly, we prove that $\mathrm{Q}'(S)$ is planar. It is clear that new edges never cross edges already in $S$, and type B edges do not cross each other by construction. Furthermore, by the order we add edges, two edges sharing a vertex do not cross. Now we will study other possibilities case by case.

  We start by proving that there is no crossing between type A edges. Suppose that two type A edges $e_1 = \{u_1, v_1\}$ and $e_2 = \{u_2, v_2\}$ cross each other, with $u_i$ an ancestor of $v_i$ of depth $d_i$ for $i=1,2$. The crossing implies that, without loss of generality, $u_1$ is an ancestor of $u_2$, which means $d_1 < d_2$. Furthermore, $v_1$ and $v_2$ have a common ancestor $u'$ of depth $d_2 + 1$, and $v_1$ must precede $v_2$ in the prefix order, or else $e_1$ and $e_2$ will not cross each other. Figure~\ref{fig:link-case}(a) illustrates this case. Let $u''$ be the ancestor of $v_1$ of depth $d_1+1$. By the construction of type A edges, either $v_1$ is the certificate of $u''$, or we have $\ell(v_1) = d_1+1$. In both cases, we have $\ell(v_1) \leq d_1+1$. Since $e_2$ is a type A edge, either $\ell(v_2)=d_2+1$, or $v_2$ is the certificate of $u'$. It is not possible that $\ell(v_2)=d_2+1$, because it will violate Condition~3 of sticky trees on $u'$, whose descendant $v_1$ precedes $v_2$ with $\ell(v_1) \leq d_1+1 < \ell(v_2)=d_2+1$, which is the depth of $u'$. It is also impossible that $v_2$ is the certificate of $u'$, since $u'$ is of depth $d_2+1$, larger than the label of $v_1$, and $v_1$ precedes $v_2$. Therefore, we reach a contradiction, which means edges of type A never cross each other.

  We now prove that there is no crossing between type A and B edges. Suppose that a type A edge $e_1$ from $u_1$ to $v_1$ crosses a type B edge $e_2$ from $u_2$ to $v_2$ with disjoint vertices. By the construction of type B edges, $v_2$ is the certificate of some vertex $v_2'$. Let $d_1$ be the depth of $u_1$, $d_2$ the depth of $v_2'$, and $w$ be the parent of $v_2'$. We know that $\ell(v_1) \leq d_1+1$ and $\ell(v_2) < d_2$. According to either $u_1$ precedes both $u_2$ and $v_2$ in the prefix order or lies in between, there are two possibilities for $e_1$ to cross $e_2$: either $v_1$ is in the sub-tree rooted at $v_2'$ and precedes $v_2$ in the prefix order, and $u_1$ is an ancestor of $w$ (see Figure~\ref{fig:link-case}(b)); or $u_1$ is in the sub-tree rooted at $v_2'$, and one of its child $u_1'$ is a common ancestor of $v_1$ and $v_2$, while $v_1$ follows $v_2$ in the prefix order (see Figure~\ref{fig:link-case}(c)). In the first case, since $v_1$ is a descendant of $v_2'$ that precedes $v_2$, which is the certificate of $v_2'$, we must have $\ell(v_1) \geq d_2$. Therefore, $d_1+1 \geq d_2$, which is impossible, since $u_1$ should be an ancestor of $w$, thus also of $v_2'$. In the second case, we have $d_1 \geq d_2$, and according to the construction of type A edges, either $\ell(v_1)=d_1+1$ or $v_1$ is the certificate of $u_1'$. It is not possible that $\ell(v_1)=d_1+1$, or else Condition~3 of sticky trees will be violated on $u_1'$ and its two descendants $v_1$ and $v_2$. However, $v_1$ cannot be the certificate of $u_1'$, since $v_2$ precedes $v_1$ in the prefix order, and we already have $\ell(v_2) < d_2 < d_1 + 1$, with $d_1+1$ the depth of $u_1'$. We also reach a contradiction for the second case, concluding that no type A edge crosses a type B edge. We thus establish that $\mathrm{Q}'(S)$ is planar.

  We finally prove that $\mathrm{Q}'(S)$ has the good number of edges to be the core of a planar triangulation. We know that $\mathrm{Q}'(S)$ has $n+1$ vertices, and let $k_\ell$ (resp. $k_r$) be the number of edges on their left (resp. right) boundary. By Euler's relation, we need to prove that there are $2n-k_\ell-k_r$ new edges. We first observe that the right boundary of $\mathrm{Q}'(S)$ is the right-most branch $P_r$ of $S$, and every node except those on $P_r$ contribute exactly one type B new edge, which makes $n-k_r$ new edges of type B. We now prove that all nodes with a non-zero label contribute each exactly one type A new edge. Let $u$ be a non-root node. If $u$ is primary, its label is not $0$, and it contributes the edge that links its parent to its certificate; if $u$ is derived with a non-zero label, it contributes the edge that links $u$ to its ancestor of depth $d-1$. Since the nodes with the label $0$ are exactly those on the left boundary, we have $n-k_\ell$ edges of type A, which makes up the good total number. We conclude that $\mathrm{Q}'(S)$ is the core of a triangulation.
\end{proof}

\begin{figure}
  \begin{center}
    \includegraphics[scale=0.85, page=10]{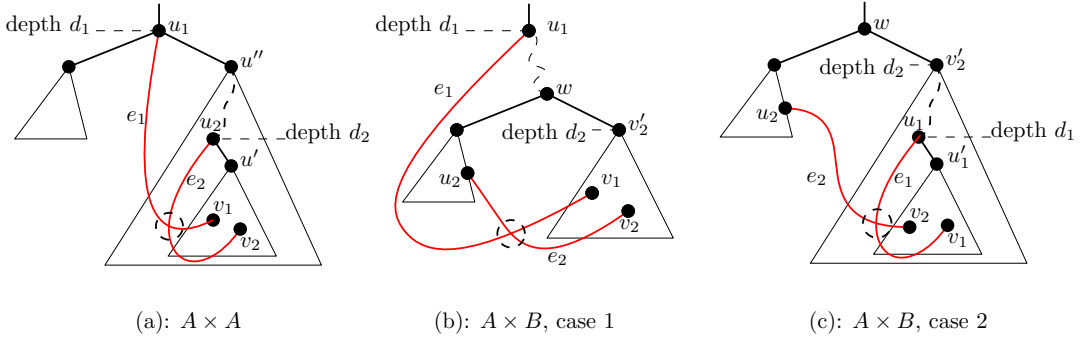}
  \end{center}
  \caption{Cases of potential crossings in triangulation reconstruction}
  \label{fig:link-case}
\end{figure}

Readers may remark that new edges added to $S$ in the core of $\mathrm{Q}(S)$ for a vertex $u$ are exactly those removed when visiting $u$ in the construction of $\mathrm{P}(\mathrm{Q}(S))$. The following theorem is thus not a surprise.

\begin{thm} \label{thm:s-t-bij}
For all $n \geq 1$, the transformation $\mathrm{P}$ is a bijection from $\mathcal{T}_n$ to $\mathcal{S}_n$, and $\mathrm{Q}$ is its inverse.
\end{thm}
\begin{proof}
  The size parameter is clearly preserved by the bijections. By Proposition~\ref{prop:t-to-s} and Proposition~\ref{prop:s-to-t}, we only need to prove that $\mathrm{Q}(\mathrm{P}(T))=T$ for any planar triangulation $T$, and $\mathrm{P}(\mathrm{Q}(S))=S$ for any sticky tree $S$. It is easy to see $\mathrm{Q}(\mathrm{P}(T))=T$, since the edges added to $\mathrm{P}(T)$ in the construction of $\mathrm{Q}(\mathrm{P}(T))$ are exactly those removed in the construction of $\mathrm{P}(T)$.

  To prove that $\mathrm{P}(\mathrm{Q}(S))=S$, we observe that the exploration tree of $\mathrm{Q}(S)$ has the same tree structure as $S$, since all edges added into $\mathrm{Q}(S)$ for a node $v$ in $S$ are removed in the construction of $\mathrm{P}(\mathrm{Q}(S))$. We then only need to prove that $\mathrm{P}(\mathrm{Q}(S))$ has the same label as $S$ on each vertex. Since $S$ has the same tree structure as the exploration tree of $\mathrm{Q}(S)$, primary vertices are correctly labeled in $\mathrm{P}(\mathrm{Q}(S))$. Let $v$ be a non-root primary node in $S$ of depth $d$, and $v'$ its certificate. If there is a descendant $w$ of $v$ with a label $d$ in $S$, by Condition~3 of sticky trees, $w$ must come before $v$ in the prefix order. Therefore, by construction, $w$ is correctly relabeled by $d$ in $\mathrm{P}(\mathrm{Q}(S))$. As a consequence, all derived nodes of $S$ with a label at least $1$ are correctly relabeled in $\mathrm{P}(\mathrm{Q}(S))$. The remaining nodes are obliged to be labeled $0$, which are exactly those on the left boundary of the core of $\mathrm{Q}(S)$. Therefore, we have $\mathrm{P}(\mathrm{Q}(S))=S$.
\end{proof}

From a more recursive point of view, Property~4 of Lemma~\ref{lem:triang-expl} can be seen as a recursive decomposition of cores of planar triangulations.

\section{Bijections to other combinatorial structures} \label{sec:bij-oth}

Other than planar bridgeless map and planar triangulations, the counting formula \eqref{eq:count} also counts other combinatorial objects, including intervals in the Tamari lattice of order $n$ \cite{ch06,BB2009intervals} and closed flows on forests with $n$ nodes \cite{chapoton-chatel-pons}. For Tamari intervals, its counting formula was first proved in \cite{ch06} using the symbolic method, then a bijection from Tamari intervals to planar triangulations was given in \cite{BB2009intervals}. For closed flows on forests, it was proved in \cite{chapoton-chatel-pons} using a bijection to interval posets, which are related to Tamari intervals. In the following, we will briefly describe direct bijections from sticky trees to these objects.

We first introduce a function on nodes of a sticky tree, which is useful in both bijections that we describe in the following. Let $S$ be a sticky tree, its \tdef{certificate-counting} function $c$ maps nodes of $S$ to values in $\mathbb{N}_+$, and for a node $u$ in $S$, its value $c(u)$ is the number of nodes whose certificate is $u$.

A \tdef{Dyck path} is a path $D$ on $\mathbb{N}^2$ composed by up steps $u=(1,1)$ and down steps $d=(-1,1)$ that starts and finishes on the $x$-axis while crossing to the negative side. It is clear that Dyck paths have even length. A Dyck path can also be seen as a word in the alphabet $\{u,d\}$. We say that the $i^{\rm th}$ up step $u_i$ is \tdef{matched} with a down step $d_j$ in $D$ if the factor $D_i$ of $D$ between $u_i$ and $d_j$ (excluding $u_i$ and $d_j$) is also a Dyck path. We denote by $\ell_D(i)$ the length of $D_i$. The \tdef{Tamari lattice} of order $n$ is a partial order $\preceq_T$ defined on the set of Dyck paths of length $2n$, where we have $D \preceq_T E$ for two Dyck paths $D$ and $E$ if and only if $\ell_D(i) \leq \ell_E(i)$ for all $1 \leq i \leq n$, and in this case, the pair $[D,E]$ is called an \tdef{interval} $[D,E]$ of the Tamari lattice of order $n$. We say that the size of $[D,E]$ is $n$ in this case. Geometrically, it means that, if for both $D$ and $E$ we draw a horizontal ray from the middle of their $i^{th}$ up step until it touches the same Dyck path again, then the ray can extend longer on $E$ than on $D$. Our definition of the Tamari lattice here is not the standard one, but is an equivalent according to Proposition~5 in \cite{bousquet-fusy-preville}. We also refer readers to \cite{bousquet-fusy-preville} for a more standard and geometric definition illustrated therein.

We now introduce the related notion of \emph{synchronized intervals} here. We first define the \tdef{type} $Type(D)$ of a Dyck path $D$ of length $2n$ as follows: it is a word $w$ of length $n-1$ in two letter $N$ and $E$ such that $w_i=N$ if the $i^{\rm th}$ up step of $D$ is followed by a down step, and $w_i=E$ otherwise. A Tamari interval $I = [D_1,D_2]$ is \tdef{synchronized} if $Type(D_1)=Type(D_2)$. Geometrically, it is equivalent to say that down steps of $D_1$ and $D_2$ occupy the same set of lines. See Figure~\ref{fig:decorated-sticky} for an example of a synchronized interval. More details of synchronized intervals can be found in \cite{PRV2014extension, nonsep}.

To obtain a Tamari interval $\mathrm{I}(S)=[\mathrm{D}(S),\mathrm{E}(S)]$ from a sticky tree $S$, we first do a counter-clockwise contour traversal of $S$, whose variation in depth gives the upper path $\mathrm{E}(S)$. Let $v_1, v_2, \ldots, v_k$ the \emph{non-root} nodes of $S$ in the prefix order, and the lower path $\mathrm{D}(S)$ is given by $\mathrm{D}(S)=ud^{c(v_1)}ud^{c(v_2)}\cdots ud^{c(v_k)}$, where $c$ is the certificate-counting function. Figure~\ref{fig:other-bij} shows an example of this bijection. We can see that $\mathrm{D}(S)$ is always a Dyck path, since the certificate of a node is always itself or its descendant, therefore, in the traversal, we always encounter first a node, then its certificate, giving more $u$ than $d$ in any prefix.

\begin{figure}
  \begin{center}
    \includegraphics[page=7,scale=0.95]{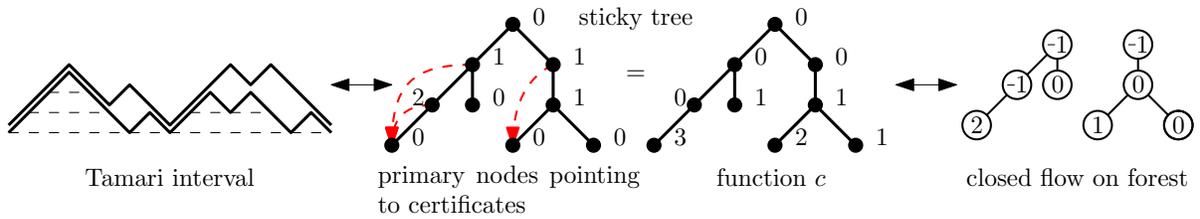}
  \end{center}
  \caption{Bijections from sticky trees to Tamari intervals and to closed flows on forests}
  \label{fig:other-bij}
\end{figure}

\begin{prop} \label{prop:s-to-i}
  For a sticky tree $S$ with $n+1$ nodes, $\mathrm{I}(S)$ is a Tamari interval of length $2n$.
\end{prop}
\begin{proof}
  The fundamental reason is that the certificate of a node $u$ is always a descendant of $u$, therefore always comes after $u$ in the prefix order. We now give a detailed proof. Let $v_1, \ldots, v_n$ be the list of \emph{non-root} nodes in $S$ in prefix order. We now fix $i$ between $1$ and $n$. Suppose that $v_i, v_{i+1}, \ldots, v_{i+k}$ are the nodes in the sub-tree of $S$ rooted at $v_i$. From the construction of $\mathrm{E}(S)$, we know that $v_j$ is first visited after the $j^{\rm th}$ up step $u_j$ in $\mathrm{E}(S)$, and we thus have $\ell_{\mathrm{E}(S)}(i) = 2k$. We now consider the portion $D_{[i,i+k]} = ud^{c(v_i)}\cdots ud^{c(v_{i+k})}$ of $\mathrm{D}(S)$ given by $v_i, \ldots, v_{i+k}$. Since the certificate of a node $w$ must be a descendant of $w$, we have $\sum_{j=i}^{i+k} c(v_j) \geq k$. Therefore, the matching step of the $i^{\rm th}$ up step $u_i$ in $\mathrm{D}(S)$ must be in $D_{[i,i+k]}$. However, there are exactly $k$ up steps in $D_{[i,i+k]}$ other than $u_i$, and we have a Dyck path $D'$ between $u_i$ and its matching step, which then contains at most $k$ step. We thus conclude that $\ell_{\mathrm{D}(S)}(i)$, which is the length of $D'$, is at most $2k$. We thus have $\ell_{\mathrm{D}(S)}(i) \leq \ell_{\mathrm{E}(S)}(i)$ for all $1 \leq i \leq n$, which means that $\mathrm{I}(S)$ is a Tamari interval.
\end{proof}

We now turn to closed flow on forests. A \tdef{forest} is an ordered list $F=(A_1, \ldots, A_k)$ of plane trees. A \tdef{flow} on a forest $F$ is a function $f$ defined on nodes of $F$, also called \tdef{inputs} on nodes, such that $f(v) \geq -1$ for every node $v$, and the \emph{outgoing rate} of each node is non-negative. The \tdef{outgoing rate} of a node $v$ is the sum of all the inputs of nodes in the sub-tree of $F$ rooted at $v$ (including $v$). We say that a flow $f$ on $F$ is \tdef{closed} if the outgoing rates of the roots of all $A_i$'s are all $0$. Given a sticky tree $S$, let $F$ be the forest obtained by deleting the root of $S$, then the function $f(v)=c(v)-1$ on nodes of $S$ based on the counting function $c(v)$ gives a closed flow on $F$. See Figure~\ref{fig:other-bij} for an example.

For a proof of the reversed direction of these two bijections from sticky trees to Tamari intervals and closed forest flows, we first observe that both bijections rely on the certificate-counting function $c$ in a simple way. Thus, we only need to prove that we can recover labels on a sticky tree from its certificate-counting function $c$. A direct proof can be given, but we prefer to see this fact as a special case of a similar result in \nonsep{} for decorated trees, as illustrated in Figure~\ref{fig:decorated-sticky}.

We first recall the notion of decorated trees from \nonsep{}. A \tdef{decorated tree} is a rooted plane tree $R$ with a label function $\ell$ that takes values \emph{only} on leaves, satisfying the following three conditions:
\begin{enumerate}
\item[1'.] For a leaf $f$ attached to a node of depth $d$, we have $-1 \leq \ell(f) \leq d-1$.
\item[2'.] For each internal node $u$ of depth $d>0$, there is at least one descendant leaf $f$ such that $\ell(f) < d-1$.
\item[3'.] For $t$ a node of depth $d$, $u$ a child of $t$ and $R_u$ a sub-tree rooted at $u$, if there is a leaf $f$ in $R_u$ with $\ell(f) = d$ (which is the depth of $t$), then all leaves in $R_u$ that come before $f$ in the prefix order have a label at least $d$.
\end{enumerate}
For an internal node $u$ of depth $d>0$, its \tdef{certificate} is the first leaf in prefix order that makes $u$ satisfies the Condition~(2'). We can see that the definition of sticky trees is reminiscent to that of decorated trees. Indeed, sticky trees can be seen as a variant of decorated trees. We need further definitions to clarify this point.

We define $\mathcal{RS}_n$ as the set of decorated trees with $n+1$ internal nodes and $n+1$ leaves, such that each internal node has a leaf as its first child in the prefix order. The upper-left corner of Figure~\ref{fig:decorated-sticky} shows an example of a decorated tree in $\mathcal{RS}_6$. For readers familiar with \nonsep{}, these decorated trees are in bijection with synchronized intervals of type $(NE)^n$, which are in turn in bijection with Tamari intervals of size $n$.

\begin{figure}
  \begin{center}
    \includegraphics[page=12,scale=0.85]{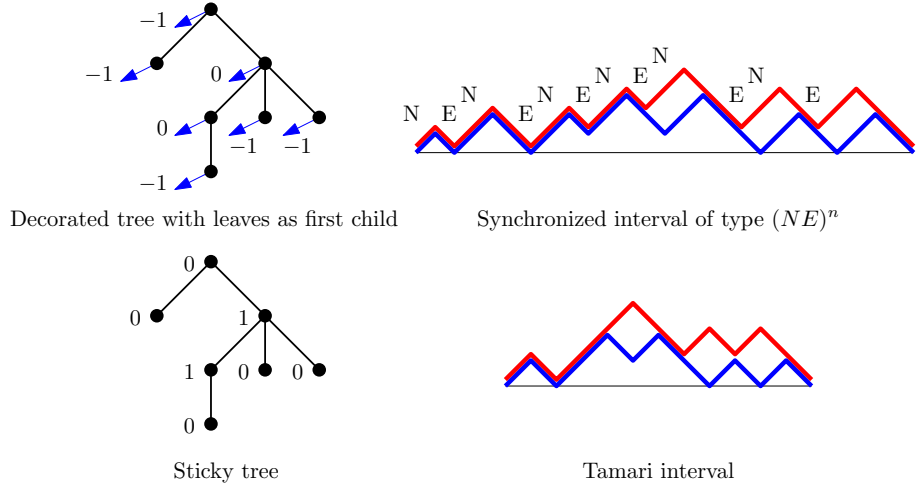}
  \end{center}
  \caption{Example of a decorated tree and its corresponding synchronized interval of type $(NE)^n$, and their corresponding sticky tree and Tamari interval}
  \label{fig:decorated-sticky}
\end{figure}

We now define the following bijection $\textrm{Ctr}$ from $\mathcal{RS}_n$ to $\mathcal{S}_n$: let $R \in \mathcal{RS}_n$, we delete all leaves in $R$ and move their labels to their parent while adding $1$ to labels to obtain $\textrm{Ctr}(R)$. This transformation is well-defined, since leaves and internal nodes in trees in $\mathcal{RS}_n$ are in one-to-one correspondence. It is also clear that $\textrm{Ctr}$ is invertible, and the inverse $\textrm{Ctr}$ is given by adding a leaf to each node as the first child, then move the labels on nodes to the leaves while subtracting $1$. We have the following proposition.

\begin{prop} \label{prop:contract-sticky}
  The transformation $\textrm{Ctr}$ is a bijection between $\mathcal{S}_n$ and $\mathcal{RS}_n$.
\end{prop}
\begin{proof}
  Let $S$ be a tree with labels on all nodes and $R$ a tree with only labels on leaves such that $S=\textrm{Ctr}(R)$. We only need to show that $R \in \mathcal{RS}_n$ if and only if $S \in \mathcal{S}_n$. The proof consists of comparing the conditions of sticky trees and of decorated trees under the transformation $\textrm{Ctr}$. From the definition of $\textrm{Ctr}$, the tree $S$ satisfies Conditions~(1) of sticky trees if and only if $R$ satisfies Conditions~(1') of decorated trees. For Condition~(2) and (2'), we only need to observe that the certificate node of a node $u$ in $S$ is exactly the parent of the certificate leaf of the corresponding node $u'$ in $R$. For Condition~(3) and (3'), we observe that (3') is a condition on a sub-tree rooted at $u$, which is a child of $t$, a node of depth $d$. Therefore, in $R$, we are dealing with labels $d$ in a sub-tree of depth $d+1$, while in $S$ it is labels $d$ in a sub-tree of depth $d$. Since labels are incremented by $1$ from $R$ to $S$, the two conditions are equivalent.
\end{proof}

As we observed in the proof of Proposition~\ref{prop:contract-sticky}, we have the following property of certificates of both sticky trees and decorated trees in $\mathcal{RS}_n$.

\begin{prop}
  Let $S$ be a sticky tree and $R=\textrm{Ctr}(S)$ its corresponding decorated tree. Suppose that $f$ is the certificate of an internal node $u$ in $R$, and the internal node $v$ is the parent of $f$. Let $u'$, $v'$ be the nodes corresponding to $u$ and $v$ in $S$. Then $v'$ is the certificate of $u'$.
\end{prop}

We now define the \tdef{certificate-counting} function $c$ of a tree $R$ in $\mathcal{RS}_n$, which is a function on leaves of $R$, and for a leaf $f$, $c(f)$ is the number of internal node $u$ in $R$ such that $f$ is the certificate of $u$. We thus have the following corollary.

\begin{coro} \label{coro:cc-equal}
  Let $S$ be a sticky tree and $R=\textrm{Ctr}(S)$, with $c_S$ and $c_R$ their certificate function respectively. We have $c_S(u) = c_R(f)$ for every node $u$ in $S$, with $f$ the first child of the corresponding internal node $u'$ in $R$.
\end{coro}

In \nonsep{}, the following proposition was implicitly proved, which states that we can recover the labels of a decorated tree by their certificate-counting functions (called ``charges'' in \nonsep{}).

\begin{prop}[See Propositions~4.6 and 4.8 in \nonsep{}]
  Let $R$ be a decorated tree. Given the tree structure and the certificate-counting function $c_R$ of $R$, we can recover leaf labels of $R$, and this is a bijection.
\end{prop}

Composing with the bijection $\textrm{Ctr}$, with Corollary~\ref{coro:cc-equal}, we have the following corollary.

\begin{coro}
  Let $S$ be a sticky tree. Given the tree structure and the certificate-counting function $c_S$ of $S$, we can recover labels of $S$, and this is a bijection.
\end{coro}

This is exactly what we need to establish bijections between sticky trees, Tamari intervals and closed flows on forests.

We observe that, since the bijections from sticky trees to Tamari intervals and closed flows on forests only rely on the certificate-counting function $c$ of sticky trees, we can construct a direct bijection from Tamari intervals to closed flows on forests. Moreover, the upper path of a Tamari interval corresponds to the shape of its sticky tree, which in turns determines the shape of the forest on which the corresponding closed flow lives. Therefore, we have an alternative bijective proof of the following theorem in \cite{chapoton-chatel-pons}.

\begin{thm}[Theorem~4.1 in \cite{chapoton-chatel-pons}]
Given a Dyck path $D$ of length $2n$, there is a forest $F(D)$ with $n$ nodes such that the number of elements $E$ smaller than $D$ in the Tamari lattice of order $n$ is the number of closed flows on $F(D)$.
\end{thm}

By composing the bijection from sticky trees to Tamari intervals with the one described in Section~\ref{sec:bij-tri}, we obtain a bijection from planar triangulations to Tamari intervals. Experimentally, this new bijection is different from the one in \cite{BB2009intervals}. Their relation is to be investigated.

We also observe that, the bijection from Tamari intervals to sticky trees, when restricted to synchronized intervals, gives sticky trees in a special form: all its internal nodes are primary. It is because, for a synchronized interval $[P,Q]$, an up step $u_i$ in $P$ is followed by a down step -- which means the corresponding node $v_i$ in the resulting sticky tree has $c(v_i) > 0$ -- if and only if its counterpart in $Q$ is also followed by a down step, which means the same node $v_i$ is a leaf. Let $\mathcal{SS}_n$ be the set of these sticky trees coming from synchronized intervals of length $2n$. We check easily that, for a tree $S \in \mathcal{SS}_n$, by removing labels on internal nodes and subtracting $1$ from those on leaves, we obtain a decorated tree $S'$, and this is a bijection. Therefore, we can also see the bijections in \nonsep{} between synchronized intervals and decorated trees as a special case of the bijections here between Tamari intervals and sticky trees.

\section{Discussion} \label{sec:dis}

We have seen bijections between sticky trees and various combinatorial objects. It is not surprising that our bijections transfer interesting statistics and structures between these objects. For instance, the number of primary nodes in sticky trees has the same distribution as the number of vertices in bridgeless planar maps and the number of vertices with a negative input in closed flows on forests. We can thus obtain non-trivial structural results with our bijections. For instance, vertical symmetry is a straight-forward involution on closed flows on forests, which means that the length of the left-most and the right-most branch have the same distribution in sticky trees, therefore the \emph{initial rise} (\textit{i.e.}, the number of initial up steps) and the \emph{final descent} (\textit{i.e.}, the number of final down steps) of upper paths also have the same distribution in Tamari intervals. This non-trivial result is an immediate corollary of results in \cite{chapoton-chatel-pons}, and it can also be seen lucidly under our bijection. We thus expect further study in this direction to reveal more hidden structures of these objects, similar to what we have done in \cite{trinity-duality} for non-separable planar maps, synchronized intervals and related objects.

Another motivation comes from intervals in the $m$-Tamari lattice. In \cite{bousquet-fusy-preville}, it has been proved that the number of intervals in the $m$-Tamari lattice can be expressed by a formula similar to those in planar map enumeration. We thus want to find a natural class of planar maps in bijection with these intervals. Although we can cast $m$-Tamari intervals as generalized Tamari intervals with a canopy of the form $(NE^m)^n$ (\textit{cf.} \cite{PRV2014extension}), the direct restriction of the bijection between non-separable planar maps and generalized Tamari intervals in \cite{nonsep} does not give a natural class of planar maps. Our work here on sticky trees can be seen as an effort to adapt the bijection in \cite{nonsep} to $m$-Tamari intervals in the special case $m=1$. It remains to see if a similar adaptation applies for general $m$.

\section*{Acknowledgements}
  We thank Guillaume Chapuy, Eric Fusy and Louis-François Préville-Ratelle for their inspiring discussions and useful comments. We also thank the anonymous referees for their comments that greatly improves the presentation of this article.

\bibliographystyle{alpha}
\bibliography{fang-bij-maps}
\end{document}